\documentclass{amsart}
\usepackage{amssymb}
\usepackage{amsmath}
\usepackage{amsfonts}

\newtheorem{theorem}{Theorem}
\theoremstyle{plain}

\newtheorem{corollary}{Corollary}

\newtheorem{lemma}{Lemma}

\newtheorem{proposition}{Proposition}
\newtheorem{remark}{Remark}

\numberwithin{equation}{section}
\input{tcilatex}

\begin{document}
\title[On the rescaled Riemannian metric of Cheeger Gromoll type]{On the
rescaled Riemannian metric of \\
Cheeger Gromoll type on the cotangent bundle}
\author{A. GEZER}
\address{Ataturk University, Faculty of Science, Department of Mathematics,
25240, Erzurum-Turkey.}
\email{agezer@atauni.edu.tr}
\author{M. ALTUNBAS}
\address{Erzincan University, Faculty of Science and Art, Department of
Mathematics, 24030, Erzincan-Turkey.}
\email{maltunbas@erzincan.edu.tr}
\subjclass[2000]{ 53C07, 53C55, 55R10.}
\keywords{Almost paracomplex structure, connection, cotangent bundle,
paraholomorphic tensor field, Riemannian metric.}

\begin{abstract}
Let $(M,g)$ be an $n-$dimensional Riemannian manifold and $T^{\ast }M$ be
its cotangent bundle equipped with a Riemannian metric of Cheeger Gromoll
type which rescale the horizontal part by a nonzero differentiable function.
The main purpose of the present paper is to discuss curvature properties of $%
T^{\ast }M$ and construct almost paracomplex Norden structures on $T^{\ast
}M.$ We investigate conditions for these structures to be para-K\"{a}hler
(paraholomorphic) and quasi-K\"{a}hler. Also, some properties of almost
paracomplex Norden structures in context of almost product Riemannian
manifolds are presented.
\end{abstract}

\maketitle

\section{\protect\bigskip \textbf{Introduction}}

\noindent Geometric structures on bundles have been object of much study
since the middle of the last century. The natural lifts of the metric $g$,
from a Riemannian manifold $(M,g)$ to its tangent or cotangent bundles,
induce new (pseudo) Riemannian structures, with interesting geometric
properties. Maybe the best known Riemannian metric $^{S}g$ on the tangent
bundle over Riemannian manifold $(M,g)$ is that introduced by Sasaki in 1958
(see \cite{Sasaki}), but in most cases the study of some geometric
properties of the tangent bundle endowed with this metric led to the
flatness of the base manifold. This metric $^{S}g$ is a standard notion in
differential geometry called the Sasaki metric. The Sasaki metric $^{S}g$
has been extensively studied by several authors and in many different
contexts. Another Riemannian metric $^{CG}g$ on the tangent bundle $TM$ had
been defined, some years before, by E. Musso and F. Tricerri \cite{Musso}
who, inspired by the paper \cite{Cheeger} of J. Cheeger and D. Gromoll,
called it the Cheeger-Gromoll metric. The metric was defined by J. Cheeger
and D. Gromoll; yet, there were E. Musso and F. Tricerri who wrote down its
expression, constructed it in a more \textquotedblright
comprehensible\textquotedblright\ way, and gave it the name. In \cite%
{Zayatuev1}(see also \cite{Zayatuev2,Zayatuev3}, B. V. Zayatuev introduced a
Riemannian metric $^{S}\overline{g}$ on the tangent bundle $TM$ given by 
\begin{eqnarray*}
^{S}\overline{g}\left( ^{H}X,^{H}Y\right) &=&fg\left( X,Y\right) , \\
^{S}\overline{g}\left( ^{H}X,^{V}Y\right) &=&^{S}\overline{g}\left(
^{V}X,^{H}Y\right) =0, \\
^{S}\overline{g}\left( ^{V}X,^{V}Y\right) &=&g\left( X,Y\right) ,
\end{eqnarray*}%
where $f>0$, $f\in C^{\infty }(M).$ For $f=1$, it follows that $^{S}%
\overline{g}=^{S}g,$ i.e. the metric $^{S}\overline{g}$ is a generalization
of the Sasaki metric $^{S}g$. Also, the authors studied the rescaled Sasaki
type metric on the cotangent bundle $T^{\ast }M$ over Riemannian manifold $%
(M,g)$ (see \cite{Gezer2}).

Almost complex Norden and almost paracomplex Norden structures are among the
most important geometrical structures which can be considered on a manifold.
Let $M_{2k}$ be a $2k$-dimensional differentiable manifold endowed with an
almost (para) complex structure $\varphi $ and a pseudo-Riemannian metric $g$
of signature $(k,k)$ such that $g(\varphi X,Y)=g(X,\varphi Y)$, i.e. $g$ is
pure with respect to $\varphi $ for arbitrary vector fields $X$ and $Y$ on $%
M_{2k}$. Then the metric $g$ is called Norden metric. Norden metrics are
referred to as anti-Hermitian metrics or $B$-metrics. They find widespread
application in mathematics as well as in theoretical physics. Many authors
considered almost (para)complex Norden structures on the tangent, cotangent
and tensor bundles \cite%
{Druta1,Gezer1,Olszak,Oproiu1,Oproiu2,Oproiu3,Oproiu4,Papag1,Papag2,Salimov2,Salimov3}%
.

Let $T^{\ast }M$ be the cotangent bundle of a Riemannian manifold $(M,g)$.
We define $r^{2}=g^{-1}\left( p,p\right) =g^{i}{}^{j}p_{i}p_{j}$ \ and put $%
\alpha =1+r^{2}$. Then the rescaled Riemannian metric of Cheeger Gromoll
type $^{CG}g_{f}$ is defined on $T^{\ast }M$ by the following three equations%
\begin{equation}
^{CG}g_{f}\left( {}^{V}\omega ,{}^{V}\theta \right) ={}\frac{1}{\alpha }%
(g^{-1}\left( \omega ,\theta \right) +g^{-1}\left( \omega ,p\right)
g^{-1}\left( \theta ,p\right) ),  \label{A1.1}
\end{equation}%
\begin{equation}
{}^{CG}g_{f}\left( {}^{V}\omega ,{}^{H}Y\right) =0,  \label{A1.2}
\end{equation}%
\begin{equation}
{}^{CG}g_{f}\left( {}^{H}X,{}^{H}Y\right) ={}fg\left( X,Y\right)
\label{A1.3}
\end{equation}%
for any $X,Y\in \Im _{0}^{1}\left( M\right) $ and $\omega ,\theta \in \Im
_{1}^{0}\left( M\right) $, where $f>0$, $f\in C^{\infty }(M)$, $g^{-1}\left(
\omega ,\theta \right) =g^{i}{}^{j}\omega _{i}\theta _{j}$ for all $\omega $%
, $\theta \in \Im _{1}^{0}\left( M\right) $ (for $f=1,$ see \cite{Agca}). In
this paper, firstly, curvature tensor of the rescaled Cheeger Gromoll type
metric $^{CG}g_{f}$ \ are presented. Secondly, we get the conditions under
which the cotangent bundle endowed with some paracomplex structures and the
rescaled Riemannian metric of Cheeger Gromoll type $^{CG}g_{f}$ is a
paraholomorphic Norden manifold. Finally, for an almost paracomplex manifold
to be an specialized almost product manifold, we give some results relation
to Riemannian almost product structure on the cotangent bundle.

Throughout this paper, all manifolds, tensor fields and connections are
always assumed to be differentiable of class $C^{\infty }$. Also, we denote
by $\Im _{q}^{p}(M)$ the set of all tensor fields of type $(p,q)$ on $M$,
and by $\Im _{q}^{p}(T^{\ast }M)$ the corresponding set on the cotangent
bundle $T^{\ast }M$. The Einstein summation convention is used, the range of
the indices $i,j,s$ being always $\{1,2,...,n\}.$

\section{Preliminaries}

\subsection{The cotangent bundle}

The cotangent bundle of a smooth $n-$dimensional Riemannian manifold may be
endowed with a structure of $2n-$dimensional smooth manifold, induced by the
structure on the base manifold. If $(M,g)$ is a smooth Riemannian manifold
of dimension $n$, we denote its cotangent bundle by $\pi :T^{\ast
}M\rightarrow M.$ A system of local coordinates $\left( U,x^{i}\right) ,%
\mathrm{\;}i=1,...,n$ in $M$ induces on ${}T^{\ast }M$ a system of local
coordinates $\left( \pi ^{-1}\left( U\right) ,\mathrm{\;}x^{i},\mathrm{\;}x^{%
\overline{i}}=p_{i}\right) ,\mathrm{\;}\overline{i}=n+i=n+1,...,2n$, where $%
x^{\overline{i}}=p_{i}$ is the components of covectors $p$ in each cotangent
space ${}T_{x}^{\ast }M,\mathrm{\;}x\in U$ with respect to the natural
coframe $\left\{ dx^{i}\right\} $.

Let $X=X^{i}\frac{\partial }{\partial x^{i}}$ and $\omega =\omega _{i}dx^{i}$
be the local expressions in $U$ of a vector field $X$ \ and a covector
(1-form) field $\omega $ on $M$, respectively. Then the vertical lift $%
^{V}\omega $ of $\omega $ and the horizontal lift $^{H}X$ of $X$ are given,
with respect to the induced coordinates, by\noindent 
\begin{equation}
^{V}\omega =\omega _{i}\partial _{\overline{i}},  \label{A2.1}
\end{equation}%
and

\begin{equation}
^{H}X=X^{i}\partial _{i}+p_{h}\Gamma _{ij}^{h}X^{j}\partial _{\overline{i}},
\label{A2.2}
\end{equation}%
where $\partial _{i}=\frac{\partial }{\partial x^{i}}$, $\partial _{%
\overline{i}}=\frac{\partial }{\partial x^{\overline{i}}}$ and $\Gamma
_{ij}^{h}$ are the coefficients of the Levi-Civita connection $\nabla $ of $%
g $.

The Lie bracket operation of vertical and horizontal vector fields on $%
T^{\ast }M$ is given by the formulas%
\begin{equation}
\left\{ 
\begin{array}{l}
{\left[ ^{H}X,^{H}Y\right] ={}^{H}\left[ X,Y\right] +^{V}\left( p\circ
R(X,Y)\right) } \\ 
{\left[ ^{H}X,^{V}\omega \right] ={}^{V}\left( \nabla _{X}\omega \right) }
\\ 
{\left[ ^{V}\theta ,^{V}\omega \right] =0}%
\end{array}%
\right.  \label{A2.3}
\end{equation}%
for any $X,$ $Y$ $\Im _{0}^{1}(M)$ and $\theta $, $\omega \in \Im
_{1}^{0}(M) $, where $R$ is the Riemannian curvature of $g$ defined by $%
R\left( X,Y\right) =\left[ \nabla _{X},\nabla _{Y}\right] -\nabla _{\left[
X,Y\right] }$ (for details, see \cite{YanoIshihara:DiffGeo})$.$

\subsection{Expressions in the adapted frame}

We insert the adapted frame which allows the tensor calculus to be
efficiently done in $T^{\ast }M.$ With the connection $\nabla $ of $g$ on $M$%
, we can introduce adapted frames on each induced coordinate neighborhood $%
\pi ^{-1}(U)$ of $T^{\ast }M$. In each local chart $U\subset M$, we write $%
X_{(j)}=\dfrac{\partial }{\partial x^{j}},$ $\theta ^{(j)}=dx^{j},$ $%
j=1,...,n.$ Then from (\ref{A2.1}) and (\ref{A2.2}), we see that these
vector fields have, respectively, local expressions 
\begin{equation*}
^{H}X_{(j)}=\frac{\partial }{\partial x^{j}}+p_{a}\Gamma _{hj}^{a}\partial _{%
\overline{h}}
\end{equation*}%
\begin{equation*}
^{V}\theta ^{(j)}=\frac{\partial }{\partial x\overline{^{j}}}
\end{equation*}%
with respect to the natural frame $\left\{ \frac{\partial }{\partial x^{j}},%
\frac{\partial }{\partial x\overline{^{j}}}\right\} $. These $2n$ vector
fields are linearly independent and they generate the horizontal
distribution of $\nabla _{g}$ and the vertical distribution of $T^{\ast }M$,
respectively. We call the set $\left\{ ^{H}X_{(j)},^{V}\theta ^{(j)}\right\} 
$ the frame adapted to the connection $\nabla $ of $g$ in $\pi
^{-1}(U)\subset T^{\ast }M$. By denoting%
\begin{eqnarray}
E_{j} &=&^{H}X_{(j)},  \label{A2.4} \\
E_{\overline{j}} &=&^{V}\theta ^{(j)},  \notag
\end{eqnarray}%
we can write the adapted frame as $\left\{ E_{\alpha }\right\} =\left\{
E_{j},E_{\overline{j}}\right\} $. The indices $\alpha ,\beta ,\gamma
,...=1,...,2n$ indicate the indices with respect to the adapted frame.

Using (\ref{A2.1}), (\ref{A2.2}) and (\ref{A2.4}), we have%
\begin{equation}
^{V}\omega =\left( 
\begin{array}{l}
0 \\ 
\omega _{j}%
\end{array}%
\right) ,  \label{A2.5}
\end{equation}%
and

\begin{equation}
^{H}X=\left( 
\begin{array}{l}
X^{j} \\ 
0%
\end{array}%
\right)  \label{A2.6}
\end{equation}%
with respect to the adapted frame $\left\{ E_{\alpha }\right\} $ (for
details, see \cite{YanoIshihara:DiffGeo}). By the straightforward
calculations, we have the lemma below.

\begin{lemma}
\label{Lemma1}The Lie brackets of the adapted frame of $T^{\ast }M$ satisfy
the following identities:%
\begin{eqnarray*}
\left[ E_{i},E_{j}\right] &=&p_{s}R_{ijl}^{\text{ \ \ }s}E_{\overline{l}}, \\
\left[ E_{i},E_{\overline{j}}\right] &=&\Gamma _{il}^{j}E_{\overline{l}}, \\
\left[ E_{\overline{i}},E_{\overline{j}}\right] &=&0
\end{eqnarray*}

where $R_{ijl}^{\text{ \ \ }s}$ denote the components of the curvature
tensor of $(M,g)$.\bigskip
\end{lemma}

\section{\protect\bigskip Curvature tensor of the rescaled Riemannian metric
of Cheeger Gromoll type}

From the equations (\ref{A1.1})-(\ref{A1.3}), by virtue of (\ref{A2.5}) and (%
\ref{A2.6}), the rescaled Cheeger Gromoll type metric $^{CG}g_{f}$ has
components with respect to the adapted frame $\left\{ E_{\alpha }\right\} $:%
\begin{equation}
^{CG}g_{f}=\left( 
\begin{array}{cc}
{fg_{ij}} & {0} \\ 
{0} & \frac{1}{\alpha }({g}^{ij}+{g}^{is}{g}^{tj}p_{s}p_{t})%
\end{array}%
\right) \text{ .}  \label{A33.1}
\end{equation}

For the Levi-Civita connection of the rescaled Cheeger Gromoll type metric $%
^{CG}g_{f}$ we give the next theorem.

\begin{theorem}
\label{propo1} Let $(M,g)$ be a Riemannian manifold and equip its cotangent
bundle $T^{\ast }M$ with the rescaled Cheeger Gromoll type metric $%
^{CG}g_{f} $. Then the corresponding Levi-Civita connection $\widetilde{%
\nabla }$ satisfies the followings:%
\begin{equation}
\left\{ 
\begin{array}{l}
i)\text{ }\widetilde{\nabla }_{E_{i}}E_{j}=\{\Gamma
_{ij}^{l}+{}^{f}A_{ij}^{l}\}E_{l}+\dfrac{1}{2}p_{s}R_{ijl}^{\text{ \ \ }s}E_{%
\overline{l}}, \\ 
ii)\text{ }\widetilde{\nabla }_{E_{i}}E_{\overline{j}}=\dfrac{1}{2f\alpha }%
p_{s}R_{.\text{ }i\text{ .}}^{l\text{ \ \ }js}E_{l}-\Gamma _{il}^{j}E_{%
\overline{l}}, \\ 
iii)\text{ }\widetilde{\nabla }_{E\overline{_{i}}}E_{j}=\dfrac{1}{2f\alpha }%
p_{s}R_{.\text{ }j\text{ .}}^{l\text{ \ \ }is}E_{l}, \\ 
iv)\text{ }\widetilde{\nabla }_{E\overline{_{i}}}E_{\overline{j}}=\{\frac{-1%
}{\alpha }(p^{i}\delta _{l}^{j}+p^{j}\delta _{l}^{i})+\frac{\alpha +1}{%
\alpha ^{2}}g^{ij}p_{l}+\frac{1}{\alpha ^{2}}p^{i}p^{j}p_{l}\}E_{\overline{l}%
}%
\end{array}%
\right.  \label{A33.2}
\end{equation}%
with respect to the adapted frame, where ${}^{f}A_{ji}^{h}$ is a tensor
field of type $(1,2)$ defined by ${}^{f}A_{ji}^{h}=\frac{1}{f}(f_{j}\delta
_{i}^{h}+f_{i}\delta _{j}^{h}-f_{.}^{m}g_{ji})$ and $p^{i}=g^{it}p_{t}$, $%
R_{.}^{k}{}_{j\text{ }.}{}^{i}{}^{s}=g^{kt}g^{im}R_{tjm}^{\text{ \ \ }s}.$
\end{theorem}

\begin{proof}
The connection \bigskip ${}\widetilde{{\nabla }}$ is characterized by the
Koszul formula:%
\begin{eqnarray*}
2^{CG}g_{f}({}\widetilde{{\nabla }}_{\widetilde{X}}\widetilde{Y},\widetilde{Z%
}) &=&\widetilde{X}(^{CG}g_{f}(\widetilde{Y},\widetilde{Z}))+\widetilde{Y}%
(^{CG}g_{f}(\widetilde{Z},\widetilde{X}))-\widetilde{Z}(^{CG}g_{f}(%
\widetilde{X},\widetilde{Y})) \\
-^{CG}g_{f}(\widetilde{X},[\widetilde{Y},\widetilde{Z}]) &+&^{CG}g_{f}(%
\widetilde{Y},[\widetilde{Z},\widetilde{X}])+\text{ }^{CG}g_{f}(\widetilde{Z}%
,[\widetilde{X},\widetilde{Y}])
\end{eqnarray*}%
for all vector fields $\widetilde{X},\widetilde{Y}$ and $\widetilde{Z}$ on $%
T^{\ast }M$. One can verify the Koszul formula for pairs $\widetilde{X}=$ $%
E_{i},E_{\overline{i}}$ and $\widetilde{Y}=$ $E_{j},E_{\overline{j}}$ and $%
\widetilde{Z}=$ $E_{k},E_{\overline{k}}$. In calculations, the formulas (\ref%
{A2.4}), Lemma \ref{Lemma1} and the first Bianchi identity for $R$ should be
applied. We omit standart calculations.
\end{proof}

Let $\tilde{X},\;\tilde{Y}\in \Im _{0}^{1}(T^{\ast }M)$. Then the covariant
derivative $\tilde{\nabla}_{\tilde{Y}}\tilde{X}$ has components

\begin{equation*}
\tilde{\nabla}_{\tilde{Y}}\tilde{X}^{\alpha }=\tilde{Y}^{\gamma }E_{\gamma }%
\tilde{X}^{\alpha }+\tilde{\Gamma}_{\gamma \beta }^{\alpha }\tilde{X}^{\beta
}\tilde{Y}^{\gamma }
\end{equation*}%
with respect to the adapted frame $\left\{ E_{\alpha }\right\} $. \noindent
Using (\ref{A2.4}), (\ref{A2.5}), (\ref{A2.6}) and (\ref{A33.2}), we have
the following proposition.

\begin{proposition}
\label{propo2}Let $(M,g)$ be a Riemannian manifold\textit{\ and }${}%
\widetilde{\nabla }$\textit{\ be the Levi-Civita connection of the cotangent
bundle }$T^{\ast }M$\textit{\ equipped with the rescaled Cheeger Gromoll
type metric ${}$}$^{CG}g_{f}$\textit{. Then }${}$ 
\begin{equation*}
\begin{array}{l}
{i)\mathrm{\;\;}{}}\widetilde{{\nabla }}{_{{}^{H}X}{}^{H}Y={}^{H}\left(
\nabla _{X}Y+^{f}A(X,Y)\right) +\frac{1}{2}{}^{V}\left( p\circ R\left(
X,Y\right) \right) ,} \\ 
{ii)\mathrm{\;\;}{}}\widetilde{{\nabla }}{_{{}^{H}X}{}^{V}\theta ={}\frac{1}{%
2f\alpha }{}^{H}\left( p\left( g^{-1}\circ R\left( \mathrm{\;},X\right) 
\widetilde{\theta }\right) \right) +^{V}\left( \nabla _{X}\theta \right) ,}
\\ 
{iii)\mathrm{\;\;}}\widetilde{{\nabla }}{_{{}^{V}\omega }{}^{H}Y=\frac{1}{%
2f\alpha }{}^{H}\left( p\left( g^{-1}\circ R\left( \mathrm{\;,}Y\right) 
\tilde{\omega}\right) \right) ,} \\ 
{iv)\mathrm{\;\;\;}{}}\widetilde{{\nabla }}{_{{}^{V}\omega }{}^{V}\theta =-%
\frac{1}{\alpha }\left( {}^{CG}g\left( {}^{V}\omega ,\gamma \delta \right)
{}^{V}\theta +{}^{CG}g_{f}\left( {}^{V}\theta ,\gamma \delta \right)
{}^{V}\omega \right) +\frac{\alpha +1}{\alpha }{}^{CG}g}_{f}{\left(
{}^{V}\omega ,{}^{V}\theta \right) \gamma \delta } \\ 
{\mathrm{\;\;\;\;\;\;\;\;\;\;\;\;\;\;\;\ }-\frac{1}{\alpha }{}^{CG}g}_{f}{%
\left( {}^{V}\omega ,\gamma \delta \right) {}^{CG}g}_{f}{\left( {}^{V}\theta
,\gamma \delta \right) \gamma \delta }%
\end{array}%
\end{equation*}%
\noindent \textit{for all} $X,Y\in \Im _{0}^{1}\left( M\right) $, $\omega
,\theta \in \Im _{1}^{0}\left( M\right) $\textit{, where }$\tilde{\omega}%
=g^{-1}\circ \omega \in \Im _{0}^{1}\left( M\right) ,$ $R\left( \mathrm{\;,}%
X\right) \tilde{\omega}\in \Im _{1}^{1}\left( M\right) ,$ $g^{-1}\circ
R\left( \mathrm{\;,}X\right) \tilde{\omega}\in \Im _{0}^{1}\left( M\right) $%
, $R$ \textit{and }$\gamma \delta $ \textit{denote respectively the
curvature tensor of }$\nabla $\textit{\ and the canonical or Liouville
vector }field\textit{\ on }$T^{\ast }M$\textit{\ with the local expression }$%
\gamma \delta =p_{i}E_{\overline{i}}$\textit{\ (for }$f=1,$ see \cite{Agca}).
\end{proposition}

\bigskip The Riemannian curvature tensor $\widetilde{R}$ of $T^{\ast }M$
with the rescaled Cheeger Gromoll type metric\textit{\ ${}$}$^{CG}g_{f}$ is
obtained from the well-known formula%
\begin{equation*}
\widetilde{R}\left( \widetilde{X},\widetilde{Y}\right) \widetilde{Z}=%
\widetilde{\nabla }_{\widetilde{X}}\widetilde{\nabla }_{\widetilde{Y}}%
\widetilde{Z}-\widetilde{\nabla }_{\widetilde{Y}}\widetilde{\nabla }_{%
\widetilde{X}}\widetilde{Z}-\widetilde{\nabla }_{\left[ \widetilde{X},%
\widetilde{Y}\right] }\widetilde{Z}
\end{equation*}%
for all $\widetilde{X},\widetilde{Y},\widetilde{Z}\in \Im _{0}^{1}(T^{\ast
}M)$. Then from Lemma \ref{Lemma1} and Theorem \ref{propo1}, we get the
following proposition.

\begin{proposition}
\label{propo3}The components of the curvature tensor $\widetilde{R}$ of of
the cotangent bundle $T^{\ast }M$ with the rescaled Cheeger Gromoll type
metric\textit{\ ${}$}$^{CG}g_{f}$ are given as follows:%
\begin{eqnarray*}
{\,{}\widetilde{R}(E}_{l},E_{i})E_{j} &=&\{R_{lij}^{m}-\frac{1}{2f\alpha }%
p_{t}p_{a}R_{lih}^{a}{R_{.}^{m}{}_{j}{}_{.}^{h}{}^{t}+}\frac{1}{4f\alpha }%
p_{t}p_{a}({R_{.}^{m}{}_{l}{}_{.}^{h}{}^{t}}R_{ijh}^{a}-{%
R_{.}^{m}{}_{i}{}_{.}^{h}{}^{t}}R_{ljh}^{a}) \\
&&+\nabla _{l}(A_{ij}^{m})-\nabla
_{i}(A_{lj}^{m})+A_{lh}^{m}A_{ij}^{h}-A_{ih}^{m}A_{lj}^{h})\}E_{m} \\
&&+\{{\frac{1}{2f}p_{t}(\nabla _{l}R_{ijm}^{t}-\nabla _{i}R_{ljm}^{t})+}%
\frac{1}{2}p_{t}(R_{lhm}^{t}A_{ij}^{h}-R_{ihm}^{t}A_{lj}^{h})\}E_{\bar{m}} \\
{{}\,{}\widetilde{R}(E}_{\bar{l}},E_{i})E_{j} &=&\{{\frac{-1}{2f\alpha }%
p_{a}\nabla _{i}R_{.}^{m}{}_{j}{}_{.}^{l}{}^{a}+}\frac{1}{2f\alpha }p_{a}({%
R_{.}^{m}{}_{h}{}_{.}^{l}{}^{a}A}_{ij}^{h}-{R_{.}^{h}{}_{j}{}_{.}^{l}{}^{a}A}%
_{ih}^{m}+\frac{f_{i}}{f}{R_{.}^{m}{}_{j}{}_{.}^{l}{}^{a}})\}E_{m} \\
&&+\{\frac{1}{2}R_{ijm}^{l}-\frac{1}{4f\alpha }p_{t}p_{a}R_{ihm}^{t}{%
R_{.}^{h}{}_{j}{}_{.}^{l}{}^{a}-}\frac{1}{2\alpha }p_{a}p^{l}R_{ijm}^{a}-%
\frac{\alpha +1}{2\alpha ^{2}}p_{a}p_{m}{R_{ij}{}_{.}^{l}{}^{a}\}E}_{\bar{m}}
\\
{{}\widetilde{R}(E}_{l},E_{\overline{i}})E_{j} &=&\{{\frac{1}{2f\alpha }%
p_{a}\nabla _{l}R_{.}^{m}{}_{j}{}_{.}^{i}{}^{a}+\frac{1}{2f\alpha }%
p_{a}(R_{.}^{h}{}_{j}{}_{.}^{i}{}^{a}A}_{lh}^{m}-{%
R_{.}^{m}{}_{h}{}_{.}^{i}{}^{a}A}_{lj}^{h}-\frac{f_{l}}{f}{%
R_{.}^{m}{}_{j}{}_{.}^{i}{}^{a}})\}E_{m} \\
&&+\{\frac{-1}{2}R_{ljm}^{i}-\frac{1}{4f\alpha }p_{t}p_{a}R_{lhm}^{a}{%
R_{.}^{h}{}_{j}{}_{.}^{i}{}^{t}+}\frac{1}{2\alpha }p_{a}p^{i}R_{ljm}^{a}-%
\frac{\alpha +1}{2\alpha ^{2}}p_{a}p_{m}{R_{lj}{}_{.}^{i}{}^{a}\}E}_{\bar{m}}
\\
{{}\widetilde{R}(E}_{\bar{l}},E_{\overline{i}})E_{j} &=&\{\frac{1}{%
4f^{2}\alpha ^{2}}p_{t}p_{a}({%
R_{.}^{m}{}_{h}{}_{.}^{l}{}^{a}R_{.}^{h}{}_{j}{}_{.}^{i}{}^{t}-R_{.}^{m}{}_{h}{}_{.}^{i}{}^{a}R_{.}^{h}{}_{j}{}_{.}^{l}{}^{t})+%
}\frac{1}{f\alpha }{R_{.}^{m}{}_{j}{}_{.}^{i}{}^{l})} \\
&&{+}\frac{1}{f\alpha ^{2}}p_{a}(p^{i}{R_{.}^{m}{}_{j}{}_{.}^{l}{}^{a}-}p^{l}%
{R_{.}^{m}{}_{j}{}_{.}^{i}{}^{a}}\}E_{m}
\end{eqnarray*}%
\begin{eqnarray*}
{\,\,\widetilde{R}(E}_{l},E_{i})E_{\bar{j}} &=&\{{\frac{1}{2f\alpha }%
p_{a}(\nabla _{l}R_{.}^{m}{}_{i}{}_{.}^{j}{}^{a}-\nabla
_{i}R_{.}^{m}{}_{l}{}_{.}^{j}{}^{a})+}\frac{1}{2f\alpha }p_{a}({%
R_{.}^{h}{}_{i}{}_{.}^{j}{}^{a}A}_{lh}^{m}-{R_{.}^{h}{}_{l}{}_{.}^{j}{}^{a}A}%
_{ih}^{m} \\
&&-\frac{f_{l}}{f}{R_{.}^{m}{}_{i}{}_{.}^{j}{}^{a}+}\frac{f_{i}}{f}{%
R_{.}^{m}{}_{l}{}_{.}^{j}{}^{a}})\}E_{m}+\{R_{ilm}^{j}+\frac{1}{4f\alpha }%
p_{t}p_{a}(R_{lhm}^{t}{R_{.}^{h}{}_{i}{}_{.}^{j}{}^{a}} \\
&&{-R_{ihm}^{a}{R_{.}^{h}{}_{l}{}_{.}^{j}{}^{t})}+}\frac{1}{\alpha }%
p_{a}p^{j}R_{l\text{ }im}^{\text{ \ \ \ }a}-\frac{\alpha +1}{\alpha ^{2}}%
p_{a}p_{m}{R_{li}{}_{.}^{j}{}^{a}\}E}_{\bar{m}} \\
{{}{}\widetilde{R}(E}_{\bar{l}},E_{i})E_{\bar{j}} &=&\{{\frac{1}{2f\alpha }%
R_{.}^{m}{}_{i}{}_{.}^{j}{}^{l}+}\frac{1}{2f\alpha ^{2}}p_{a}(p^{l}{%
R_{.}^{m}{}_{i}{}_{.}^{j}{}^{a}+}p^{i}{R_{.}^{m}{}_{i}{}_{.}^{l}{}^{a})+}%
\frac{1}{4f^{2}\alpha ^{2}}p_{a}p_{t}{%
R_{.}^{m}{}_{h}{}_{.}^{l}{}^{a}R_{.}^{h}{}_{i}{}_{.}^{j}{}^{t}\}E}_{m} \\
{\widetilde{R}(E}_{l},E_{\overline{i}})E_{\bar{j}} &=&\{{\frac{-1}{2f\alpha }%
R_{.}^{m}{}_{l}{}_{.}^{ji}{}+}\frac{1}{2f\alpha ^{2}}p_{a}(p^{i}{%
R_{.}^{m}{}_{l}{}_{.}^{j}{}^{a}+}p^{j}{R_{.}^{m}{}_{l}{}_{.}^{i}{}^{a})-}%
\frac{1}{4f^{2}\alpha ^{2}}p_{a}p_{t}{%
R_{.}^{m}{}_{h}{}_{.}^{i}{}^{a}R_{.}^{h}{}_{l}{}_{.}^{j}{}^{t}\}E}_{m} \\
{\widetilde{R}(E}_{\bar{l}},E_{\overline{i}})E_{\bar{j}} &=&\{\frac{\alpha
^{2}+\alpha +1}{\alpha ^{3}}(g^{ij}\delta _{m}^{l}-g^{jl}\delta _{m}^{i})+%
\frac{\alpha +2}{\alpha ^{3}}(g^{lj}p^{i}p_{m}-g^{ij}p_{l}p_{m}) \\
&&+\frac{\alpha -1}{\alpha ^{3}}(\delta _{m}^{i}p^{l}p^{j}-\delta
_{m}^{l}p^{i}p^{j}){\}E}_{\bar{m}}
\end{eqnarray*}%
with respect to the adapted frame $\left\{ E_{\alpha }\right\} $ (for $f=1,$
see \cite{Agca})$.$
\end{proposition}

\section{\protect\bigskip para-K\"{a}hler (or paraholomorphic) Norden
structures on $T^{\ast }M$}

An almost paracomplex manifold is an almost product manifold $%
(M_{2k},\varphi )$, $\varphi ^{2}=id$, $\varphi \neq \pm id$ such that the
two eigenbundles $T^{+}M_{2k}$ and $T^{-}M_{2k}$ associated to the two
eigenvalues +1 and -1 of $\varphi $, respectively, have the same rank. Note
that the dimension of an almost paracomplex manifold is necessarily even.
This structure is said to be integrable if the matrix $\varphi =(\varphi
_{j}^{i})$ is reduced to the constant form in a certain holonomic natural
frame in a neighborhood $U_{x}$ of every point $x\in M_{2k}$. On the other
hand, an almost paracomplex structure is integrable if and only if one can
introduce a torsion-free linear connection such that $\nabla \varphi =0$. A
paracomplex manifold is an almost paracomplex manifold $(M_{2k},\varphi )$
such that the G-structure defined by the affinor field $\varphi $ is
integrable. Also it can be give another-equivalent-definition of paracomplex
manifold in terms of local homeomorphisms in the space $R^{k}(j)=\left\{
(X^{1},...,X^{k})/X^{i}\in R(j),\mathrm{\;}i=1,...,k\right\} $ and
paraholomorphic changes of charts in a way similar to \cite{CruForGad:1995}
(see also \cite{Vish:DiffGeo}), i.e. a manifold $M_{2k}$ with an integrable
paracomplex structure $\varphi $ is a realization of the paraholomorphic
manifold $M_{k}(R(j))$ over the algebra $R(j)$.

A tensor field $\omega $ of type $(0,q)$ is called pure tensor field with
respect to $\varphi $ if 
\begin{equation*}
\omega (\varphi X_{1},X_{2},...,X_{q})=\omega (X_{1},\varphi
X_{2},...,X_{q})=...=\omega (X_{1},X_{2},...,\varphi X_{q})
\end{equation*}%
for any $X_{1},...,X_{q}\in \Im _{0}^{1}(M_{2k}).$ The real model of a
paracomplex tensor field $\mathop{\omega }\limits^{\ast }$ on $M_{k}(R(j))$
is a $(0,q)-$tensor field on $M_{2k}$ which being pure with respect to $%
\varphi $. Consider an operator $\Phi _{\varphi }:\Im
_{q}^{0}(M_{2k})\rightarrow \Im _{q+1}^{0}(M_{2k})$ applied to the pure
tensor field $\omega $ by (see \cite{YanoAko:1968})%
\begin{equation*}
(\Phi _{\varphi }\omega )(X,Y_{1},Y_{2},...,Y_{q})=(\varphi X)(\omega
(Y_{1},Y_{2},...,Y_{q}))-X(\omega (\varphi Y_{1},Y_{2},...,Y_{q}))
\end{equation*}%
\begin{equation*}
+\omega ((L_{Y_{1}}\varphi )X,Y_{2},...,Y_{q})+...+\omega
(Y_{1},Y_{2},...,(L_{Y_{q}}\varphi )X),
\end{equation*}%
where $L_{Y}$ denotes the Lie differentiation with respect to $Y$. \noindent
Let $\varphi $ be a (an almost) paracomplex structure on $M_{2k}$ and $\Phi
_{\varphi }\omega =0$, the (almost) paracomplex tensor field $%
\mathop{\omega
}\limits^{\ast }$ on $M_{k}(R(j))$ is said to be (almost) paraholomorphic
(see \cite{Kruchkovich:1972}, \cite{Tachibana}, \cite{YanoAko:1968}). Thus a
(an almost) paraholomorphic tensor field $\mathop{\omega }\limits^{\ast }$
on $M_{k}(R(j))$ is realized on $M_{2k}$ in the form of a pure tensor field $%
\omega $, such that%
\begin{equation*}
(\Phi _{\varphi }\omega )(X,Y_{1},Y_{2},...,Y_{q})=0
\end{equation*}%
for any $X,Y_{1},...,Y_{q}\in \Im _{0}^{1}(M_{2k})$. Therefore, the tensor
field $\omega $ on $M_{2k}$ is also called a (an almost) paraholomorphic
tensor field.

An almost paracomplex Norden manifold $(M_{2k},\varphi ,g)$ is defined to be
a real differentiable manifold $M_{2k}$ endowed with an almost paracomplex
structure $\varphi $ and a Riemannian metric $g$ satisfying Nordenian
property (or purity condition)%
\begin{equation*}
g(\varphi X,Y)=g(X,\varphi Y)
\end{equation*}%
for any $X,Y\in \Im _{0}^{1}(M_{2k})$. Manifolds of this kind are referred
to as anti-Hermitian and B-manifolds (see \cite%
{Druta1,Oproiu1,Oproiu2,Oproiu3,Oproiu4,Papag1,Papag2,
SalimovIscanEtayo:2007, Vishnevskii:2002, Vish:DiffGeo}). If $\varphi $ is
integrable, we say that $(M_{2k},\varphi ,g)$ is a paracomplex Norden
manifold. A paracomplex Norden manifold $(M_{2k},\varphi ,g)$ is a
realization of the paraholomorphic manifold $(M_{k}(R(j)),\mathop{g}%
\limits^{\ast })$, where $\mathop{g}\limits^{\ast }=(\mathop{g}%
\limits_{uv}^{\ast }),\quad u,v=1,...,k$ is a paracomplex metric tensor
field on $M_{k}(R(j))$.

In a paracomplex Norden manifold, a paracomplex Norden metric $g$ is called
paraholomorphic\textit{\ }if%
\begin{equation}
(\Phi _{\varphi }g)(X,Y,Z)=0  \label{A3.1}
\end{equation}%
for any $X,Y,Z\in \Im _{0}^{1}(M_{2k})$. \noindent The paracomplex Norden
manifold with paraholomorphic Norden metric $(M_{2k},\varphi ,g)$ is called
a paraholomorphic Norden manifold.

In \cite{SalimovIscanEtayo:2007}, Salimov and his collaborators have been
proven that for an almost paracomplex manifold with Norden metric $g$, the
condition $\Phi _{\varphi }g=0$ is equivalent to $\nabla \varphi =0$, where $%
\nabla $ is the Levi-Civita connection of $g$. By virtue of this point of
view, paraholomorphic Norden manifolds are similar to K\"{a}hler manifolds.

V. Cruceanu defined in \cite{Cru} an almost paracomplex structure on $%
T^{\ast }M$ as follows:%
\begin{eqnarray}
J(^{H}X) &=&-^{H}X  \label{A3.5} \\
J({}^{V}\omega ) &=&^{V}\omega  \notag
\end{eqnarray}%
for any $X\in \Im _{0}^{1}\left( M\right) $ and $\omega \in \Im
_{1}^{0}\left( M\right) $. One can easily check that the metric $^{CG}g_{f}$
is pure with respect to the almost paracomplex structure $J$. Hence we state
the following theorem.

\begin{theorem}
\label{Theo1}Let $(M,g)$ be a Riemannian manifold and $T^{\ast }M$ be its
cotangent bundle equipped with the rescaled Cheeger Gromoll type metric $%
^{CG}g_{f}$ and the paracomplex structure $J$. The triple $(T^{\ast
}M,J,^{CG}g_{f})$ is an almost paracomplex Norden manifold.
\end{theorem}

We now give conditions for the rescaled Cheeger Gromoll type metric $%
^{CG}g_{f}$ to be paraholomorphic with respect to the almost paracomplex
structure $J$. Using defination of the rescaled Cheeger Gromoll type metric $%
^{CG}g_{f}$ and the almost paracomplex structure $J$ and by using the fact
that ${}^{V}\omega {}^{V}(g^{-1}(\theta ,\sigma ))=0$ and $%
{}^{H}X{}^{V}(fg(Y,Z))={}^{V}(X(fg(Y,Z)))$ we calculate 
\begin{eqnarray*}
(\Phi _{J}{}^{CG}g_{f})(\tilde{X},\tilde{Y},\tilde{Z}) &=&(J\tilde{X}%
)({}^{CG}g_{f}(\tilde{Y},\tilde{Z}))-\tilde{X}(^{CG}g_{f}(J\tilde{Y},\tilde{Z%
})) \\
&+&{}^{CG}g_{f}((L_{\tilde{Y}}J)\tilde{X},\tilde{Z})+{}^{CG}g_{f}(\tilde{Y}%
,(L_{\tilde{Z}}J)\tilde{X})
\end{eqnarray*}%
for all $\tilde{X},\tilde{Y},\tilde{Z}\in \Im _{0}^{1}(T^{\ast }M)$. For
pairs $\tilde{X}=^{H}X,^{V}\omega $, $\widetilde{Y}=^{H}Y,^{V}\theta $ and $%
\widetilde{Z}=^{H}Z,{}^{V}\sigma $, then we get%
\begin{eqnarray}
(\Phi _{J}{}^{CG}g_{f})({}^{H}X,{}^{V}\theta ,{}^{H}Z)
&=&2^{CG}g_{f}({}^{V}\theta ,^{V}(p\circ R(X,Z))  \label{A3.6} \\
(\Phi _{J}{}^{CG}g_{f})({}^{H}X,{}^{H}Y,{}^{V}\sigma )
&=&2^{CG}g_{f}(^{V}(p\circ R(X,Y),^{V}\sigma ).  \notag
\end{eqnarray}%
and the others is zero. Therefore, we have the following result.

\begin{theorem}
\label{Theo2}Let $(M,g)$ be a Riemannian manifold and let $T^{\ast }M$ be
its cotangent bundle equipped with the rescaled Cheeger Gromoll type metric $%
^{CG}g_{f}$ and the paracomplex structure $J$. The triple $\left( T^{\ast
}M,J,{}^{CG}g_{f}\right) $ is a para-K\"{a}hler-Norden (paraholomorphic
Norden) manifold if and only if $M$ is flat.
\end{theorem}

\begin{remark}
Let $(M,g)$ be a Riemannian manifold and let $T^{\ast }M$ be its cotangent
bundle equipped with the rescaled Cheeger Gromoll type metric $^{CG}g_{f}.$
The diagonal lift $^{D}\gamma $ of $\gamma \in \Im _{1}^{1}(M)$ to $T^{\ast
}M$ is defined by the formulas%
\begin{eqnarray*}
^{D}\gamma ^{H}X &=&^{H}(\gamma X) \\
^{D}\gamma ^{V}\omega &=&-^{V}(\omega \circ \gamma )
\end{eqnarray*}%
for any $X\in \Im _{0}^{1}\left( M\right) $ and $\omega \in \Im
_{1}^{0}\left( M\right) .$ The diagonal lift $^{D}I$ of the identity tensor
field $I\in \Im _{1}^{1}(M)$ has the following properties%
\begin{eqnarray*}
^{D}I^{H}X &=&^{H}X \\
^{D}I^{V}\omega &=&-^{V}\omega
\end{eqnarray*}%
and satisfies $(^{D}I)^{2}=I_{T^{\ast }M}$. Thus, $^{D}I$ is an almost
paracomplex structure. Also, the rescaled Cheeger Gromoll type metric $%
^{CG}g_{f}$ is pure with respect to $^{D}I$, e.i. the triple $\left( T^{\ast
}M,^{D}I,{}^{CG}g_{f}\right) $ is an almost paracomplex Norden manifold.
Finally, by using $\Phi -$operator, we can say that the rescaled Cheeger
Gromoll type metric $^{CG}g_{f}$ is paraholomorphic with respect to $^{D}I$
if and only if $M$ is flat.
\end{remark}

The following remark follows directly from Proposition \ref{propo3}.

\begin{remark}
The cotangent bundle $\left( T^{\ast }M,{}^{CG}g_{f}\right) $ is never flat.
\end{remark}

As is known that the almost paracomplex Norden structure is a specialized
Riemannian almost product structure on a Riemannian manifold. The theory of
Riemannian almost product structures was initiated by K. Yano in \cite{Yano}%
. The classification of Riemannian almost-product structure with respect to
their covariant derivatives is described by A.M. Naveira in \cite{Naveira}.
This is the analogue of the classification of almost Hermitian structures by
A. Gray and L. Hervella in \cite{Gray}. Having in mind these results, M.
Staikova and K. Gribachev obtained a classification of the Riemannian almost
product structures, for which the trace vanishes (see \cite{Staikova}).
There are lots of physical applications involving a Riemannian almost
product manifold. Now we shall give some applications for almost paracomplex
Norden structures in context of almost product Riemannian manifolds.

\textit{4.1. }Let us recall almost product Riemannian manifolds. If an $n$%
-dimensional Riemannian manifold $M$, endowed with a positive definite
Riemannian metric $g$, admits a non-trivial tensor field $F$ of type $(1.1)$
such that 
\begin{equation*}
F^{2}=I
\end{equation*}%
and 
\begin{equation*}
g(FX,Y)=g(X,FY)
\end{equation*}%
for every vector fields $X,Y\in \Im _{0}^{1}(M)$, then $F$ is called an
almost product structure and $(M,F,g)$ is called an almost product
Riemannian manifold. An integrable almost product Riemannian manifold with
structure tensor $F$ is called a locally product Riemannian manifold. If $F$
is covariantly constant with respect to the Levi-Civita connection $\nabla $
of $g$ which is equivalent to $\Phi _{F}g=0$, then ($M,F,g)$ is called a
locally decomposable Riemannian manifold.

Now consider the almost product structure ${J}$ defined by (\ref{A3.5}) and
the Levi-Civita connection $\widetilde{\nabla }$ given by Proposition \ref%
{propo1}. We define a tensor field of type $(1,2)$ on $T^{\ast }M$ by%
\begin{equation*}
\widetilde{S}(\widetilde{X},\widetilde{Y})=\frac{1}{2}\{(\widetilde{\nabla }%
_{J\widetilde{Y}}J)\widetilde{X}+J((\widetilde{\nabla }_{\widetilde{Y}}J)%
\widetilde{X})-J((\widetilde{\nabla }_{\widetilde{X}}J)\widetilde{Y})\}
\end{equation*}%
for all $\widetilde{X},\widetilde{Y}\in \Im _{0}^{1}(T^{\ast }M)$. Then the
linear connection 
\begin{equation}
\overline{\nabla }_{\widetilde{X}}\widetilde{Y}=\widetilde{\nabla }_{%
\widetilde{X}}\widetilde{Y}-\widetilde{S}(\widetilde{X},\widetilde{Y})
\label{A3.8}
\end{equation}%
is an almost product connection on $T^{\ast }M$ (for almost product
connection, see \cite{Leon}).

\begin{theorem}
Let $(M,g)$ be a Riemannian manifold and let $T^{\ast }M$ be its cotangent
bundle equipped with the rescaled Cheeger Gromoll type metric $^{CG}g_{f}$
and the almost product structure $J$. Then the almost product connection $%
\overline{\nabla }$ constructed by the Levi-Civita connection $\widetilde{%
\nabla }$ of the rescaled Cheeger Gromoll type metric $^{CG}g_{f}$ and the
almost product structure $J$ is as follows:%
\begin{equation}
\left\{ 
\begin{array}{l}
i)\text{ }\overline{\nabla }_{^{H}X}^{\text{ \ \ \ \ }H}Y=\text{ }%
^{H}(\nabla _{X}Y)+^{H}(^{f}A(X,Y)) \\ 
ii)\text{ }\overline{\nabla }_{^{H}X}^{\text{ \ \ \ \ }V}\theta =\text{ }%
^{V}(\nabla _{X}\theta ), \\ 
iii)\text{ }\overline{\nabla }_{^{V}\omega }^{\text{ \ \ \ \ }H}Y=\dfrac{3}{%
2f\alpha }\text{ }{^{H}\left( p\left( g^{-1}\circ R\left( \mathrm{\;,}%
Y\right) \tilde{\omega}\right) \right) }, \\ 
iv)\text{ }\overline{\nabla }_{^{V}\omega }^{\text{ \ \ \ \ }V}\theta ={-%
\frac{1}{\alpha }\left( {}^{CG}g\left( {}^{V}\omega ,\gamma \delta \right)
{}^{V}\theta +{}^{CG}g_{f}\left( {}^{V}\theta ,\gamma \delta \right)
{}^{V}\omega \right) +\frac{\alpha +1}{\alpha }{}^{CG}g}_{f}{\left(
{}^{V}\omega ,{}^{V}\theta \right) \gamma \delta } \\ 
\text{ \ \ \ \ \ \ \ \ \ \ \ \ \ \ \ \ \ \ \ \ }{-\frac{1}{\alpha }{}^{CG}g}%
_{f}{\left( {}^{V}\omega ,\gamma \delta \right) {}^{CG}g}_{f}{\left(
{}^{V}\theta ,\gamma \delta \right) \gamma \delta .}%
\end{array}%
\right.  \label{A3.9}
\end{equation}
\end{theorem}

Denoting by $\overline{T}$, the torsion tensor of $\overline{\nabla }$, we
have from (\ref{A3.5}), (\ref{A3.8}) and (\ref{A3.9})%
\begin{eqnarray*}
\overline{T}(^{V}\omega ,^{V}\theta ) &=&0, \\
\overline{T}(^{V}\omega ,^{H}Y) &=&\dfrac{3}{2f\alpha }\text{ }{^{H}\left(
p\left( g^{-1}\circ R\left( \mathrm{\;,}Y\right) \tilde{\omega}\right)
\right) }, \\
\overline{T}(^{H}X,^{H}Y) &=&-^{V}(p\circ R(X,Y)).
\end{eqnarray*}%
Hence we have the theorem below

\begin{theorem}
Let $(M,g)$ be a Riemannian manifold and let $T^{\ast }M$ be its cotangent
bundle. The almost product connection $\overline{\nabla }$ constructed by
the Levi-Civita connection $\widetilde{\nabla }$ of the rescaled Cheeger
Gromoll type metric $^{CG}g_{f}$ and the almost product structure $J$ is
symmetric if and only if $M$ is flat.
\end{theorem}

As is well-known, if there exists a symmetric almost product connection on $%
M $ then the almost product structure $J$ is integrable \cite{Leon}. The
converse is also true \cite{Fujimoto}. Thus we get the following conclusion.

\begin{corollary}
Let $(M,g)$ be a Riemannian manifold and $T^{\ast }M$ be its tangent bundle
equipped with the rescaled Cheeger Gromoll type metric $^{CG}g_{f}$ and the
almost product structure $J$. The triple $(T^{\ast }M,J,^{CG}g_{f})$ is a
locally product Riemannian manifold if and only if $M$ is flat.
\end{corollary}

Similarly, let us consider the almost product structure $^{D}I$ and the
Levi-Civita connection $\widetilde{\nabla }$ of the rescaled Cheeger Gromoll
type metric $^{CG}g_{f}$, Another almost product connection can be
constructed.

If $J$ is covariantly constant with respect to the Levi-Civita connection $%
\widetilde{\nabla }$ of the rescaled Cheeger Gromoll type metric $^{CG}g_{f}$
which is equivalent to $\Phi _{J}^{\text{ }CG}g_{f}=0$, then $(T^{\ast
}M,J,^{CG}g_{f})$ is called a locally decomposable Riemannian manifold. In
view of Theorem \ref{Theo2}, we have the following.

\begin{corollary}
Let $(M,g)$ be a Riemannian manifold and $T^{\ast }M$ be its cotangent
bundle equipped with the rescaled Cheeger Gromoll type metric $^{CG}g_{f}$
and the almost product structure $J$. The triple $(T^{\ast }M,J,^{CG}g_{f})$
is a locally decomposable Riemannian manifold if and only if $M$ is flat.
\end{corollary}

\textit{4.2.} Let $(M_{2k},\varphi ,g)$ be a non-integrable almost
paracomplex manifold with a Norden metric. An almost paracomplex Norden
manifold $(M_{2k},\varphi ,g)$ is a quasi-para-K\"{a}hler--Norden manifold,
if $\underset{X,Y,Z}{\sigma }g((\nabla _{X}\varphi )Y,Z)=0$, where $\sigma $
is the cyclic sum by three arguments \cite{Manev}. In \cite{Salimov4}, the
authors proved that $\underset{X,Y,Z}{\sigma }g((\nabla _{X}\varphi )Y,Z)=0$
is equivalent to $(\Phi _{\varphi }g)(X,Y,Z)+(\Phi _{\varphi
}g)(Y,Z,X)+(\Phi _{\varphi }g)(Z,X,Y)=0.$ We compute%
\begin{equation*}
A(\tilde{X},\tilde{Y},\tilde{Z})=(\Phi _{J}{}^{CG}g_{f})(\tilde{X},\tilde{Y},%
\tilde{Z})+(\Phi _{J}{}^{CG}g_{f})(\widetilde{Y},\widetilde{Z},\widetilde{X}%
)+(\Phi _{J}{}^{CG}g_{f})(\widetilde{Z},\widetilde{X},\widetilde{Y})
\end{equation*}%
for all $\widetilde{X},\widetilde{Y},\widetilde{Z}\in \Im _{0}^{1}(T^{\ast
}M).$ By means of (\ref{A3.6}), we have $A(\tilde{X},\tilde{Y},\tilde{Z})=0$
for all $\widetilde{X},\widetilde{Y},\widetilde{Z}\in \Im _{0}^{1}(T^{\ast
}M).$ Hence we state the following theorem.

\begin{theorem}
\label{Theo5}Let $(M,g)$ be a Riemannian manifold and $T^{\ast }M$ be its
cotangent bundle equipped with the rescaled Cheeger Gromoll type metric $%
^{CG}g_{f}$ and the almost paracomplex structure $J$ defined by (\ref{A3.5}%
). The triple $(T^{\ast }M,J,^{CG}g_{f})$ is a quasi-para-K\"{a}hler--Norden
manifold.
\end{theorem}

O. Gil-Medrano and A.M. Naveira proved that both distributions of the almost
product structure on the Riemannian manifold ($M,F,g)$ are totally geodesic
if and only if $\underset{X,Y,Z}{\sigma }g((\nabla _{X}\varphi )Y,Z)=0$ for
any $X,Y,Z\in \Im _{0}^{1}(M)$ \cite{Gil}$.$ As a consequence of Theorem \ref%
{Theo5}, we have the following.

\begin{corollary}
Both distributions of the almost product Riemannian manifold $(T^{\ast }M,$ $%
J,^{CG}g_{f})$ are totally geodesic.
\end{corollary}

\textit{4.3.} Let $F$ be an almost product structure and $\nabla $ be a
linear connection on an $n$-dimensional Riemannian manifold $M.$ The product
conjugate connection $\nabla ^{(F)}$ of $\nabla $ is defined by%
\begin{equation*}
\nabla _{X}^{(F)}Y=F(\nabla _{X}FY)
\end{equation*}%
for all $X,Y\in \Im _{0}^{1}(M).$ If $(M,F,g)$ is an almost product
Riemannian manifold, then ($\nabla _{X}^{(F)}g)(FY,FZ)=(\nabla _{X}g)(Y,Z)$,
i.e. $\nabla $ is a metric connection with respect to $g$ if and only if $%
\nabla ^{(F)}$ is so. From this, we can say that if $\nabla $ is the
Levi-Civita connection of $g$, then $\nabla ^{(F)}$ is a metric connection
with respect to $g$ \cite{Blaga}$.$ The family of metric connections on a
Riemannian manifold $M$ which have the same geodesics as the Levi-Civita
connection is a distinguished class among the family of all connections on $%
M $. This family attracted the attention of E. Cartan in the early 20th
century. Since then, many mathematicians have been concerned with its study.
These connections arise in a natural way in theoretical and mathematical
physics. For example, such a connection is of particular interest in string
and superstring theory \cite{Agricola,Strominger}.

By the almost product structure ${J}$ defined by (\ref{A3.5}) and the
Levi-Civita connection $\widetilde{\nabla }$ given by Proposition \ref%
{propo1}, we write the product conjugate connection $\widetilde{\nabla }%
^{(J)}$ of $\widetilde{\nabla }$ as follows:%
\begin{equation*}
\widetilde{\nabla }_{\widetilde{X}}^{(J)}\widetilde{Y}=J(\widetilde{\nabla }%
_{\widetilde{X}}J\widetilde{Y})
\end{equation*}%
for all $\widetilde{X},\widetilde{Y}\in \Im _{0}^{1}(T^{\ast }M).$ Also note
that $\widetilde{\nabla }^{(J)}$ is a metric connection of the rescaled
Cheeger Gromoll type metric $^{CG}g_{f}$. The standart calculations give the
following theorem.

\begin{theorem}
Let $(M,g)$ be a Riemannian manifold and let $T^{\ast }M$ be its cotangent
bundle equipped with the rescaled Cheeger Gromoll type metric $^{CG}g_{f}$
and the almost product structure $J.$ Then the product conjugate connection
(or metric connection) $\widetilde{\nabla }^{(J)}$ is as follows: 
\begin{equation*}
\begin{array}{l}
{i)\mathrm{\;\;}{}}\widetilde{{\nabla }}{_{{}^{H}X}{}^{H}Y={}^{H}\left(
\nabla _{X}Y+^{f}A(X,Y)\right) -\frac{1}{2}{}^{V}\left( p\circ R\left(
X,Y\right) \right) ,} \\ 
{ii)\mathrm{\;\;}{}}\widetilde{{\nabla }}{_{{}^{H}X}{}^{V}\theta =-{}\frac{1%
}{2f\alpha }{}^{H}\left( p\left( g^{-1}\circ R\left( \mathrm{\;},X\right) 
\widetilde{\theta }\right) \right) +^{V}\left( \nabla _{X}\theta \right) ,}
\\ 
{iii)\mathrm{\;\;}}\widetilde{{\nabla }}{_{{}^{V}\omega }{}^{H}Y=\frac{1}{%
2f\alpha }{}^{H}\left( p\left( g^{-1}\circ R\left( \mathrm{\;,}Y\right) 
\tilde{\omega}\right) \right) ,} \\ 
{iv)\mathrm{\;\;\;}{}}\widetilde{{\nabla }}{_{{}^{V}\omega }{}^{V}\theta =-%
\frac{1}{\alpha }\left( {}^{CG}g\left( {}^{V}\omega ,\gamma \delta \right)
{}^{V}\theta +{}^{CG}g_{f}\left( {}^{V}\theta ,\gamma \delta \right)
{}^{V}\omega \right) +\frac{\alpha +1}{\alpha }{}^{CG}g}_{f}{\left(
{}^{V}\omega ,{}^{V}\theta \right) \gamma \delta } \\ 
{\mathrm{\;\;\;\;\;\;\;\;\;\;\;\;\;\;\;\ }-\frac{1}{\alpha }{}^{CG}g}_{f}{%
\left( {}^{V}\omega ,\gamma \delta \right) {}^{CG}g}_{f}{\left( {}^{V}\theta
,\gamma \delta \right) \gamma \delta .}%
\end{array}%
\end{equation*}
\end{theorem}

The relationship between curvature tensors $R_{\nabla }$ and $R_{\nabla
^{(F)}}$ of the connections $\nabla $ and $\nabla ^{(F)}$ is follows: $%
R_{\nabla ^{(F)}}(X,Y,Z)=F(R_{\nabla }(X,Y,FZ)$ for all $X,Y,Z\in \Im
_{0}^{1}(M)$ \cite{Blaga}$.$ By means of the almost product structure ${J}$
defined by (\ref{A3.5}) and Proposition \ref{propo3}, From $\widetilde{R}_{%
\widetilde{\nabla }^{(J)}}(\widetilde{X},\widetilde{Y},\widetilde{Z})=J(%
\widetilde{R}_{\widetilde{\nabla }}(\widetilde{X},\widetilde{Y},J\widetilde{Z%
})$, components of the curvature tensor $\widetilde{R}_{\widetilde{\nabla }%
^{(J)}}$ of the product conjugate connection (or metric connection) $%
\widetilde{\nabla }^{(J)}$ can easily be computed. Lastly, by using the
almost product structure $^{D}I$, another metric connection of the rescaled
Cheeger Gromoll type metric $^{CG}g_{f}$ can be constructed.

\end{document}